\documentclass[12pt]{amsart}
\usepackage{amsfonts, amssymb, amsmath, amsthm, fullpage, hyperref}
\usepackage{url}
\usepackage{enumerate}

\newtheorem{thrm}{Theorem}[section]
\newtheorem{lem}[thrm]{Lemma}
\newtheorem{prop}[thrm]{Proposition}
\newtheorem{cor}[thrm]{Corollary}
\theoremstyle{definition}
\newtheorem{definition}[thrm]{Definition}
\newtheorem{question}{Question}
\newtheorem{remark}[thrm]{Remark}
\numberwithin{equation}{section}
\newtheorem{example}{Example}[section]

\newcommand{\R}{\mathbb R}
\newcommand{\N}{\mathbb N}
\newcommand{\C}{\mathbb{C}}
\newcommand{\Z}{\mathbb Z}

\newcommand{\p}{\partial}

\newcommand{\ds}{\displaystyle}

\newcommand{\la}{\lambda}
\newcommand{\ve}{\varepsilon}

\newcommand{\ov}{\overline}
\DeclareMathOperator{\spa}{span}
\DeclareMathOperator{\tr}{tr}
\DeclareMathOperator{\re}{Re}
\DeclareMathOperator{\im}{Im}

\author{Marcin Bownik}
\address{Department of Mathematics, University of Oregon}
\email{mbownik@uoregon.edu}

\author{Li-An Daniel Wang}
\address{Department of Mathematics and Statistics, Sam Houston State University}
\email{danielwang.hads@gmail.com}

\date{\today}

\thanks{The first author was partially supported by NSF grant DMS-1956395. He would like to thank Harmut F\"uhr for helpful discussions. The second author was supported by the departmental funds from Sam Houston State University, and would like to thank Pierre Portal for discussion on this topic at ICMAT in 2013.}


\begin{document}

\title[Hardy Spaces and PDE]{A PDE Characterization of Anisotropic Hardy Spaces}

\begin{abstract} We obtain a differential characterization for the anisotropic Hardy space $H_A^p$ by identifying it with a parabolic Hardy space associated with a general continuous group. This allows $H_A^p$ to be defined using a parabolic differential equation of 
Cald\'eron and Torchinsky. We also provide a classification of dilations corresponding to equivalent anisotropic Hardy spaces with respect to linear transformations. 
\end{abstract}
\maketitle

\section{Introduction and Main Results}
Classically, Hardy spaces $H^p$ can be characterized as boundary values to the heat or Poisson equations whose associated maximal operators are in $L^p$, and these kernels are fundamental solutions to the heat and Dirichlet equations, respectively. Even as the work of Fefferman and Stein \cite{MR0447953} led to the versatile employment of the grand maximal function, the study of Hardy spaces associated with various differential operators continue unabated, with each setting introducing its challenges as well as limitations in the various equivalent formulations of a classical Hardy space. To name just a few reference works, we have Hardy spaces associated with an operator $L$ satisfying a pointwise heat kernel bounds \cite{MR2163867}, satisfying an off-diagonal estimates \cite{auscher-martell06}, associated to divergence form elliptic operators \cite{MR2481054}, the Davies-Gaffney estimates \cite{MR2868142}, and reinforced off-diagonal estimates \cite{MR3108869}. These are further generalized in the settings of weighted Hardy spaces \cite{MR3624406, MR3815287}, Gaussian Hardy spaces \cite{Portal}, and variable exponent Hardy \cite{MR3938814} among many other contributions.

Amongst all this activity, the anisotropic Hardy space \cite{MR1982689} appears to be an outlier. With its formulation grounded in maximal, atomic, and molecular characterizations, and the fact that the scale of the dilation employed in these characterizations is discrete, the anisotropic setting has no obvious characterization from a differential operator standpoint. Our objective is to answer this question: Given an anisotropic Hardy space $H_A^p$ associated with an anisotropic dilation matrix $A$,  is there a partial differential equation (PDE) whose fundamental solution can be used to define $H_A^p$? Alternatively, can we characterize it as $H_L^p$ for some differential operator $L$? As we will see, the answer to the first question is yes, but given the parabolic nature of this PDE, the answer to the second question is no.  

If a PDE does exist to characterize the anisotropic Hardy space, then we must first find a setting in which the dilation scale is continuous. A natural setting turns out to be the parabolic Hardy space of Calder{\'o}n and Torchinsky \cite{MR0417687, MR0450888}, where the use of a continuous group of dilations $\{ A_t \}_{t > 0}$, satisfying $A_s A_t = A_{st}$, provides this needed structure, as well as a PDE we can immediately use. Strictly speaking, the original form of this group is not flexible enough to accommodate the geometry of an anisotropic dilation, where the natural geometric object are ellipsoids of changing eccentricities. Our first task, in light of this, is to consider a wider class of continuous groups that can capture the essential information from a dilation matrix in the anisotropic setting. This is the content of Theorem \ref{thm:expansive}. Next, we will show that every anisotropic dilation matrix $A$ is associated with a unique continuous group that represents the same geometric information. This is the content of Theorem \ref{thm:Construct-A_t}. Lastly, we see that the Hardy spaces associated with both types of dilations are precisely the same space. This is the content of Theorem \ref{thm:Hp-Hp}.

Theorem \ref{thm:Hp-Hp} also answers the question we posed. Once the anisotropic Hardy space $H_A^p$ is identified with the parabolic Hardy space $H_{\{ A_t \}}^p$, we have a PDE characterization of $H_A^p$. Let $D_t f(x) = f(A_t x)$ be the dilation operator induced by continuous group. Then we consider the initial value problem that seeks to find $u : \R_+^{n+1} \rightarrow \C$ so that
	\begin{align}\label{eq:CT}
		\begin{cases}
		\ds \frac{\p u}{\p t} = \frac{1}{t} (D_t^{-1} \Delta D_t) u & (x,t) \in \R_+^{n+1}= \R^n \times \R_+, \\
		\ds u(x, 0) = f(x) & x\in \R^n.
		\end{cases}
	\end{align}
A tempered distribution $f$ is in $H_A^p$ exactly when the solution $u(x, t) = f \ast \tilde \Phi_t (x)$ satisfies the regularity condition that
	\[ \sup_{\tilde{\rho}(x - y) < t} |(\tilde{\Phi}_t \ast f) (x)| \in L^p. \]
where $\tilde \Phi_t(x)=t^{-1}\Phi(A_t^{-1}x)$ is the fundamental solution to the PDE \eqref{eq:CT}, and $\tilde{\rho}$ is the quasi-norm associated with $\{ A_t \}$, see Proposition \ref{fund}.

The rest of the paper is devoted to a classification of anisotropic Hardy with respect to different choices of expansive dilations. This line of research was initiated by the first author in \cite{MR1982689} and extended to anisotropic Hardy spaces with variable anisotropy by Dekel, Petrushev, and Weissblat \cite{dpw}. Recently, Cheshmavar and F\"{u}hr \cite{chesh-fuhr} have given a classification of anisotropic Besov spaces \cite{besov} associated with anisotropic dilation matrices by describing when two such  matrices induce the same scale of Besov spaces. At the same time they have found an incorrect statement in the corresponding classification theorem for anisotropic Hardy space \cite[Theorem 10.3]{MR1982689}, which we correct and expand. In particular, we obtain a refinement of classification of anisotropic Hardy spaces, which are equivalent 
with respect to linear transformations.

The paper is organized as follows. In Section \ref{S2} we provide the anisotropic and parabolic settings. In Section \ref{S3} we prove our main theorems, Theorem \ref{thm:Construct-A_t} and Theorem \ref{thm:Hp-Hp}, linking anisotropic and parabolic Hardy spaces. In Section \ref{S4} we provide the classification of anisotropic Hardy spaces by correcting and expanding the earlier results in \cite{MR1982689}. Lastly, in Section \ref{S5} we make further comments and observations about PDE \eqref{eq:CT}, its related differential operator $L$, and where the anisotropic Hardy space fits in the overall literature of Hardy spaces of various operators.

For the remainder of the paper, we will use $c$ and $C$ as general constants that may depend on the dimension $n$ or underlying geometry (anisotropic dilation matrix $A$ or expansive group $\{ A_t \}$), but do not depend on the function $f$ in question. The Schwartz class is denoted by $\mathcal{S}$, the tempered distributions denoted by $\mathcal{S}'$, and a generic test function $\varphi \in \mathcal{S}$ will be chosen so that $\int \varphi \neq 0$.

\section{Anisotropic and Parabolic Hardy Spaces}\label{S2}

\subsection{Anisotropic Setting} We now introduce the anisotropic anisotropic setting from \cite{MR1982689}, and the facts needed to state and prove our results. We say a real-valued $n \times n$ matrix $A$, is an anisotropic dilation matrix if all of its eigenvalues $\lambda$, real or complex, satisfy $|\lambda| > 1$. We can construct a (non-unique) homogeneous quasi-norm, that is, a measurable mapping $\rho_A : \R^n \rightarrow [0, \infty)$ with a constant $c$ satisfying:
        \begin{center}
        \begin{tabular}{llll}
        $\rho_A (x) = 0$
            &   $\iff$    &   $x = 0$, & \\
        $\rho_A (Ax) = b\rho(x)$
            &   for all         &   $x \in \R^n$, &   where $b:=|\det A|$, \\
        $\rho_A(x + y) \leq c(\rho_A(x) + \rho_A(y))$  &  for all & $x, y \in \R^n$. & \\
        \end{tabular}
        \end{center}
We denote $(A, \rho_A)$ to mean $\rho_A$ is a quasi-norm associated to $A$, and with $dx$ denoting the Lebesgue measure, the triplet $(\R^n, dx, \rho_A)$ is a space of homogeneous type.

To illustrate the difference between a quasi-norm and the Euclidean norm, recall that the Euclidean ball $B(x, r)$, centered at $x \in \R^n$ with radius $r$, has the nice property that whenever $r_1 < r_2$, we have $B(x, r_1) \subset B(x, r_2)$. But for a dilation matrix $A$, we do not expect $B(x, r) \subset A(B(x, r))$. Instead, one can construct `canonical' ellipsoids $\{ B_k \}_{k \in \Z}$, associated with $A$, such that for all $k \in \Z$, if we define $B_{k + 1} = A(B_k)$, then we have nested ellipsoids $B_k \subseteq B_{k + 1}$, and $|B_k| = b^k$. These nested ellipsoids will serve as the basic geometric object in the anisotropic setting. Moreover, we can use these ellipsoids to define a `step' norm associated with $A$ as follows:
    \begin{align*}
    \rho_A(x) =
            \begin{cases}
            b^j     &\textrm{ if } x \in B_{j + 1} \backslash B_j \\
            0       &\textrm{ if } x = 0.
            \end{cases}
    \end{align*}
By setting $\omega$ to be the smallest integer so that $2B_0 \subset A^{\omega} B_0 = B_{\omega}$, $\rho_A$ is a quasi-norm with the triangle inequality constant $c = b^{\omega}$. While any two quasi-norms associated with $A$ will give the same anisotropic structure {\cite[Lemma 2.4]{MR1982689}}, the step norm will be our `canonical' norm, denoted by $\rho_A$.


%


A general quasi-norm (without being associated to a dilation) is a mapping $\rho : \R^n \rightarrow [0, \infty)$ satisfying $\rho(x) = 0$ exactly when $x = 0 \in \R^n$, and for some $c > 0$, satisfies the inequality $\rho(x + y) \leq c (\rho(x) + \rho(y))$ for all $x, y \in \R^n$. Then if $\rho_1$ and $\rho_2$ be two quasi-norms on $\R^n$, we say they are equivalent if there exists $C > 0$ such that for all $x \in \R^d$,
	\[ \frac{1}{C} \rho_1 (x) \leq \rho_2 (x) \leq C \rho(x). \]
\begin{definition}\label{d21}
We say two dilations $(A_1, \rho_1)$ and $(A_2, \rho_2)$ are equivalent if their associated quasi-norms $\rho_1$ and $\rho_2$ are equivalent.
\end{definition}

We are now in a position to define Hardy spaces adapted to the geometry of dilations. We denote $\mathcal{S}$ as the Schwartz class, and $\mathcal{S}'$ the space of tempered distributions. For a dilation $A$, $\varphi \in \mathcal{S}$, and $k \in \Z$, we denote the anisotropic dilation by $\varphi_k (x) = b^{-k} \varphi(A^{-k} x)$. We have four maximal functions, corresponding to their classical counterparts, any one of which can be used to define $H_A^p$. These are  radial and non-tangential maximal functions and the corresponding grand maximal functions, defined for $f \in \mathcal{S}'$, respectively,
    \begin{align*}
    M_{\varphi}^0 f(x)
    	&= \sup_{k \in \Z} |f \ast \varphi_k (x)|, \\
	M_{\varphi} f(x)	
		&= \sup_{k \in \Z} \ \sup_{\rho(x - y) < b^{-k}} |f \ast \varphi_k (y)|, 
\\
	M^0_{\mathcal S_N} f(x) 
		&= \sup_{\varphi \in \mathcal{S}_N} M^0_{\varphi} f(x),		
		\\
	M_{\mathcal S_N} f(x) 
		&= \sup_{\varphi \in \mathcal{S}_N} M_{\varphi} f(x).
		\end{align*}
Here, $\mathcal{S}_N$, $N \in \N$, denotes the set of all $\varphi \in \mathcal{S}$ such that
	\[ \| \varphi \|_{\mathcal{S}_N} = \sup_{x \in \R^n} \sup_{|\alpha| \leq N}   |\p^{\alpha} \varphi(x)| \rho_A(x)^N  \leq 1. \]

\begin{thrm} {\cite[Proposition 3.10 and Theorem 7.1]{MR1982689}}
\label{sf}
 Let $A$ be a dilation, $p \in (0, \infty)$, $\varphi \in \mathcal{S}$ be such that $\int \varphi \neq 0$, and $N \in \N$ be sufficiently large. If $f \in \mathcal{S}'$, then the following are equivalent:
	\begin{align*}
	\ M_{\varphi}^0 f \in L^p, \qquad M_{\varphi} f \in L^p, \qquad M^0_{\mathcal S_N} f \in L^p, \qquad M_{\mathcal S_N} f \in L^p.
	\end{align*}
In this case, we say $f \in H_A^p$.
\end{thrm}

We will also use the following fact.

\begin{thrm}{\cite[Theorem 10.5]{MR1982689}}
\label{eqq}
Let $A_1$ and $A_2$ be two dilations. Then, $A_1$ and $A_2$ are equivalent if and only the corresponding anisotropic Hardy spaces coincide $H^p_{A_1}=H^p_{A_2}$ for some $0<p \le 1$. In such case they coincide for all $0<p \le 1$.
\end{thrm}

The anisotropic setting was motivated by wavelet theory, thus the parameter associated with dilation $A$ is the discrete $k \in \Z$. This causes an immediate issue when one seeks to find a differential operator $L$ or a semigroup $\{ T_t \}_{t>0}$ that would characterize $H_A^p$. That is, $T_t f$ solves the Cauchy problem $\frac{\p u}{\p t} = Lu$ and $u(x, 0) = f(x)$, with $Lu = \lim_{t \rightarrow 0^+} \frac{T_t u - u}{t}$. 
To overcome this obstacle, we will instead relate the anisotropic dilation matrix  with continuous groups $\{ A_t \}_{t > 0}$.  While this setting has been studied by a number of authors, our approach is informed by  Calder{\'o}n and Torchinsky \cite{MR0417687, MR0450888} and Stein and Wainger \cite{stein1978}.

\subsection{Parabolic Setting}
A continuous group $\{ A_t \}_{t > 0}$ is a collection of linear operators $A_t : \R^n \rightarrow \R^n$ such that for all $s, t > 0$, it satisfies the algebraic identity $A_t A_s = A_{st}$ and the mapping $t\mapsto A_t x$ is continuous for all $x\in \R^n$. This guarantees the existence of a unique $n \times n$ matrix $P$, which we call the generator of $\{ A_t \}$, such that
	\[ A_t = \exp(P \ln t), \qquad t>0. \]
Conversely, given any real-valued $n \times n$ matrix $P$, we can construct a continuous group given by the above exponential formula.

In \cite{MR0417687}, the continuous group considered carries the requirement that there are $c_1, c_2 \geq 1$ such that such that for all $t > 1$ and $x \in \R^n$,
	\begin{equation}\label{CT}
	t^{c_1} |x| \leq |A_t x| \leq t^{c_2} |x|,
	\end{equation}
which forces the Euclidean ball to be invariant under $A_t$ for all $t>1$. However, this is not sufficient to capture the setting of an anisotropic dilation, where such invariance might fail for $t>1$ close to $1$.
Instead, the natural geometric objects are ellipsoids of changing {eccentricities} whose major axes do not stay fixed. Because of this, we will work with a more general collection of semigroups $\{ A_t \}$, characterized by the following theorem.

\begin{thrm}\label{thm:expansive} Given a continuous group $\{ A_t \}_{t > 0}$ with generator $P$, the following are equivalent.
	\begin{enumerate}	
	\item[(a)] For all $x \in \R^n$, $\ds \lim_{t \rightarrow 0} |A_t x| = 0$.
	
	\item[(b)] All eigenvalues $\lambda$ of $P$ satisfy $\operatorname{Re}(\lambda) > 0$.
	
	\item[(c)] There exists $t_0 > 1$ such that for all $x \in \R^n \setminus \{0\}$, $|A_{t_0} x| > |x|$. 
\end{enumerate}
\end{thrm}
\begin{definition} A continuous group $\{ A_t \}_{t > 0}$ that satisfies the conditions from Theorem \ref{thm:expansive} is called an expansive continuous group.
\end{definition}

\begin{remark} This characterization has appeared numerous times in literature, under various names. In Stein and Wainger \cite[Part II]{stein1978}, these groups were simply called dilations, and it was stated that condition $(a)$ implied $(b)$, while in \cite[page 126]{MR2473135}, it was stated that $(b)$ implied $(a)$. While the proof of either direction only requires elementary arguments in linear algebra, we provide them here for reference. While we will not need condition (c) in the rest of the paper, we include it for completeness.
\end{remark}

\begin{proof}[Proof of Theorem \ref{thm:expansive}]
We first establish the implication $(a) \Rightarrow (b)$. Let $\lambda$ be an eigenvalue of $P$. We write $\lambda = \lambda_r + i \lambda_i$, where $\lambda_r$ and $\lambda_i \in \R$ are the real and complex parts of $\lambda$, respectively. Let $x_{\lambda} \in \C^n$ be the associated eigenvector, also potentially with complex entries, so we write
		\[ x_{\lambda} = x_r + ix_i, \]
	where $x_r$ and $x_i$ are both vectors in $\R^n$. Using the relationship $A_t = \exp(P \ln t)$ for $t > 0$, we have
		\begin{align*}
		|A_t x_{\lambda}|
		&= |\exp(P \ln t) x_{\lambda}| = |\exp(\lambda \ln t) x_{\lambda}| = |t^{\lambda} x_{\lambda}| = |t^{\lambda_r} t^{i\lambda_i} x_{\lambda}| = t^{\lambda_r} |x_{\lambda}|.
		\end{align*}
	We used the fact that for any $t > 0$, $|t^{i{\lambda_i}}| = 1$. Since $A_t x_{\lambda} \rightarrow 0$ as $t\to 0$, we have $\lambda_r > 0$.

	Next, we establish $(b) \Rightarrow (c)$.  Suppose that all eigenvalues of $P$ have positive real parts. Define the matrix $B=A_{1/2}=\exp(-P \ln 2)$. If $\lambda$ is an eigenvalue of $P$, then $2^{-\lambda}$ is an eigenvalue of $B$. Hence, all eigenvalues $\lambda$ of $B$ satisfy $|\lambda|<1$. By the spectral radius formula $\lim_{k\to \infty} ||B^k||^{1/k}<1$. Hence, there exists $k\in \N$ such that $||B^k||<1$. Let $t_0=2^k$. Since $A_{1/t_0}=B^k$, we have
	\[
	|x| \le ||A_{1/t_0}|| |A_{t_0} x| = ||B^k|| |A_{t_0} x| \qquad\text{for }x\in \R^n.
	\]
Hence, $(c)$ follows.
	
Finally, we establish $(c) \implies (a)$. Suppose that $(c)$ holds for some $t_0$. We will first show that there exists $c_0 > 1$ so that for all $x \in \R^n$,
		\begin{align}\label{def:expansive-c0}
		|A_{t_0} x| \ge  c_0 |x|.
		\end{align}
Indeed, the function $x \mapsto |A_{t_0}x|$ defined on the unit sphere $S^{n-1} \subset \R^n$ achieves a minimum value $c_0$. By $(c)$ we deduce that $c_0>1$. Hence, by the homogeneity we have \eqref{def:expansive-c0}. Applying \eqref{def:expansive-c0} recursively we have for any $k\in \N$,
\[
|A_{(t_0)^{-k}} x| \le (c_0)^{-1} |A_{(t_0)^{-k+1}} x| \le \ldots \le (c_0)^{-k} |x|.
\]
Letting $k\to \infty$ yields (a).
\end{proof}

\subsection{Homogeneous quasi-norms}
In analogy to discrete anisotropic setting we adapt the following definition of a homogeneous quasi-norm.

\begin{definition} Let $\{ A_t \}_{t > 0}$ be an expansive continuous group. We say $\tilde{\rho}: \R^n \rightarrow [0, \infty)$ is a homogeneous quasi-norm with respect to $\{ A_t \}$ if there exists $c>0$ such that
    \begin{center}
        \begin{tabular}{llll}
        $\tilde \rho(x) = 0$
            &   $\iff$    &   $x = 0$,  \\
        $\tilde \rho (A_t x) = t\tilde \rho(x)$
            &   for all         &   $x \in \R^n$, $t>0$, \\
        $\tilde \rho(x + y) \leq c(\tilde\rho(x) + \tilde\rho(y))$  &  for all & $x, y \in \R^n$. & \\
        \end{tabular}
        \end{center}
\end{definition}

The following construction yields a particularly useful example of a homogeneous quasi-norm. Let $P$ be a generator of an expansive group $\{ A_t \}_{t > 0}$. Then, there exists a positive definite symmetric matrix $B$ that can be defined algebraically \cite[Lemma 1.2]{MR0417687} by 
\begin{equation}\label{alg}
BP + P^*B = \mathbf I,
\end{equation} 
or by the integral \cite[Proposition 1-7]{stein1978}
	\[ B = \int_0^{\infty} \exp(-tP^*)\exp(-t P) dt. \]
The matrix $B$ satisfies the identity
\[	\frac{d}{dt} \langle B A_t x, A_t x \rangle = \frac{1}{t} \langle (BP + P^*B) A_t x, A_t x \rangle = \frac{1}{t} \langle A_t x, A_t x \rangle>0.
\]
Hence, for any $x \in \R^n$, the quantity $\langle B A_t x, A_t x \rangle$ is strictly increasing with respect to $t$. This allows the following construction of a homogeneous norm \cite[Definition 1-8]{stein1978}. 
Because of this property, we call $B$ the norm-inducing matrix. 

\begin{prop}[{\cite[Proposition 1-9]{stein1978}}]
 Let $\{ A_t \}_{t > 0}$ be an expansive continuous group with generator $P$. Let $B$ be a positive definite symmetric matrix satisfying \eqref{alg}. For $x \ne 0$ we define $\tilde{\rho}(x) = t$ to be the unique $t>0$ such that $\langle B A_{t^{-1}} x, A_{t^{-1}}x \rangle  = 1$. For $x=0$ we set $\tilde \rho(x)=0$. Then, $\tilde \rho$ is a homogeneous quasi-norm with respect to $\{A_t\}$, which 
is $C^\infty$ on $\R^n \setminus \{0\}$.
\end{prop}

With respect to the Fourier transform, let $\{ A_t^* \}_{t>0}$ be the adjoint of the continuous group, which is itself another continuous group with $P^*$ as its generator and with $B^*$ as its norm-inducing matrix. 
We denote $\tilde{\rho}_*$ to be the associated quasi-norm. If $A_t$ are diagonal matrices, then $\tilde{\rho}_* = \tilde{\rho}$.

With this class of continuous groups, we extend the definition of a parabolic Hardy space from \cite{MR0417687, MR0450888} verbatim, and denote such a Hardy space $H_{\{ A_t \}}^p$, emphasizing we are using the whole group $\{ A_t \}$ to define the Hardy space. Fix such a group, and with it, generator $P$, trace $\gamma = \tr(P)$, and quasi-norm $\tilde{\rho}$. Then the parabolic non-tangential maximal operator, associated with $\varphi \in \mathcal{S}$, is given by
	\[ \tilde{M}_{\varphi} f(x) = \sup_{t > 0} \ \sup_{\tilde{\rho}(x - y) < t} |f \ast \tilde{\varphi}_t (y)| \]
with the parabolic dilation $\tilde{\varphi}_t (y) = t^{-\gamma} \varphi(A_t^{-1} y)$. The presence of tilde in $\tilde M_{\varphi}$ and $\tilde{\varphi}_t$ is meant to distinguish between discrete anisotropic setting and continuous parabolic setting.

We extend the parabolic Hardy space to our setting of expansive continuous groups with no change. This definition is originally stated using non-tangential maximal operators of all apertures $a > 0$, but for brevity we will only use the aperture $a = 1$.

\begin{definition} \cite[Definition 1.1]{MR0450888}  Let $\{ A_t \}_{t > 0}$ be an expansive continuous group. If $f \in \mathcal{S}'$ and $\varphi \in \mathcal{S}$ with $\int \varphi \neq 0$, then we say $f \in H_{\{ A_t \}}^p$ if the non-tangential maximal operator $\tilde{M}_{\varphi} f \in L^p$.
In this case, we set the norm $\| f \|_{H_{\{ A_t \}}^p} = \| \tilde{M}_{\varphi} f \|_{L^p}$.
\end{definition}

The following theorem was shown in \cite[Theorem 1.2]{MR0450888} under the assumption that a continuous group of dilations $\{A_t\}_{t>0}$ satisfies \eqref{CT}. However, as we will see later it also holds for all expansive groups.

\begin{thrm} Every choice of $\varphi \in \mathcal{S}$, with $\int \varphi \neq 0$ will result in the same space $H_{\{ A_t \}}^p$.
\end{thrm}

\section{Connection Between Anisotropic and Parabolic Setting}\label{S3}
Having established the anisotropic and parabolic settings, the following result establishes their close relationship. Theorem \ref{thm:Construct-A_t} is inspired by the work of Cheshmavar and F\"uhr \cite{chesh-fuhr} on classification of anisotropic Besov spaces.

\begin{thrm}\label{thm:Construct-A_t} Suppose $A$ is an anisotropic dilation matrix. That is, all eigenvalues $\lambda$ of $A$ satisfy $|\lambda|>1$. Then there exists a unique continuous group of dilations $\{A_t= \exp(P\ln t)\}_{t>0}$, such that:
	\begin{enumerate}
	\item Its generator $P$ has all positive eigenvalues and $\tr(P)=1$, and
	\item $A$ is equivalent to $A_t$ for all $t>1$.
\end{enumerate}
More precisely, if $\{A'_t = \exp(P' \ln t)\}_{t>0}$ is another one-parameter group of dilations with a generator $P'$ satisfying (i) and $A$ is equivalent to $A'_t$ for some $t>1$, then $P=P'$. In this case, we say $\{ A_t \}_{t>0}$ is the continuous group associated with the anisotropic dilation $A$. 
\end{thrm}
\begin{remark} Due to our choice of the generator $P$, we have $\tr(P) = 1$, so $\det A_t = t$, and the dilation of a function $f$, with respect to the continuous group, takes the form $\tilde{f}_t (x) = t^{-1} f(A_t^{-1} x)$. For the rest of the paper, we will always take our generator $P$ to have trace 1.
\end{remark}

In the proof of Theorem \ref{thm:Construct-A_t} we need to use the following three lemmas from \cite{MR1982689} and \cite{chesh-fuhr} about equivalence of expansive dilations, see Definition \ref{d21}.

\begin{lem}[{\cite[Lemma 10.2]{MR1982689}}] \label{eq}
Let $A$ and $B$ be two expansive dilations. Then, $A$ and $B$ are equivalent if and only if 
\[
\sup_{k\in \Z} ||A^k B^{- \lfloor \epsilon k \rfloor}|| <\infty,
\qquad\text{where }\epsilon = \frac{\ln |\det A|}{\ln |\det B|}.
\]
\end{lem}

\begin{lem}[{\cite[Lemma 7.6]{chesh-fuhr}}] \label{cf0}
Let $A$ and $B$ be expansive matrices of the form $A=\exp(t X)$ and $B=\exp( s X)$ for some matrix $X$ and $s,t>0$. Then, $A$ and $B$ are equivalent.
\end{lem}

\begin{lem}[{\cite[Theorem 7.9(a)]{chesh-fuhr}}] \label{cf2}
Let $A$ and $B$ be expansive matrices having only positive eigenvalues and satisfying $\det A = \det B$. Then, $A$ and $B$ are equivalent if and only if $A=B$.
\end{lem}

In addition, we will need the following strengthening of a lemma due to Cheshmavar and F\"uhr \cite[Lemma 7.7]{chesh-fuhr}.

\begin{lem}\label{l6.7}
Let $A$ be an expansive matrix. Then, there exists an expansive matrix $B$ such that
\begin{enumerate}
\item $A$ is equivalent to $B$,
\item $B$ has all positive eigenvalues,
\item $\det B = |\det A|$, and
\item
for all $r>1$ and $m=1,2,\ldots$ we have
\begin{equation}\label{dim}
\sum_{|\la|=r} \dim \ker(A - \la \mathbf I)^m =
\dim \ker(B - r \mathbf I)^m.
\end{equation}
\end{enumerate}
More precisely, there is a one-to-one correspondence between blocks in a complex Jordan normal form of $A$ and blocks in a Jordan normal form of $B$ such that each Jordan block of $A$ for an eigenvalue $\lambda \in \C$ corresponds to a Jordan block of $B$ of the same size for an eigenvalue $|\lambda|$.
\end{lem}

\begin{proof}
There exists an invertible matrix $S \in GL(n,\R)$ such that $S^{-1}AS$ has real Jordan normal form. That is, $S^{-1}AS$ is a block diagonal matrix consisting of Jordan blocks corresponding to either real or complex conjugate eigenvalues of $A$. We shall define the matrix $B$ such that $S^{-1}BS$ is a block diagonal matrix where each Jordan block of $S^{-1}AS$ is replaced by a certain matrix as follows.

If $\lambda \in \C \setminus \R $ is a complex eigenvalue of $A$, then the corresponding real Jordan block is $2k \times 2k$ matrix of the form
\begin{equation}\label{jon}
J=\begin{bmatrix} M_\lambda & \mathbf I_2 &  & & \\
&  M_\lambda & \mathbf I_2 &  & \\
& & \ddots & \ddots & \\
& & & M_\lambda
\end{bmatrix}
\qquad\text{where }
M_\lambda= \begin{bmatrix} \re \lambda & \im \lambda \\
-\im \lambda & \re \lambda
\end{bmatrix}, \
\mathbf I_2 =  M_1=\begin{bmatrix} 1& 0 \\
0 &1 \end{bmatrix}.
\end{equation}
That is, $J$ is obtained from two complex Jordan blocks of size $k$ for conjugate eigenvalues $\lambda$ and $\ov{\lambda}$. Write $\lambda = |\lambda|\omega$, $|\omega|=1$. Then, $J$ can be written as a product of two commuting factors
\[
J=\begin{bmatrix} M_{\omega} & &  & & \\
&  M_{\omega} &  &  & \\
& & \ddots &  & \\
& & & M_{\omega}
\end{bmatrix}\begin{bmatrix} M_{|\lambda|} & M_{\ov \omega} &  & & \\
&  M_{|\lambda|} & M_{\ov \omega} &  & \\
& & \ddots & \ddots & \\
& & & M_{|\lambda|}
\end{bmatrix}.
\]
Let $D_1, D_2$ denote these factors. Then, $D_1$ is an isometry, whereas $D_2$ has  only one eigenvalue $|\lambda|$. We claim that the Jordan normal form of $D_2$ consists of two blocks each of size $k$. 

Indeed, let $T= D_2-|\lambda| \mathbf I_{2k}$. Then, an easy calculation shows that $T$ is a product of two commuting factors
\[
T=
\begin{bmatrix} M_{\ov \omega} &  &  & &\\
&  M_{\ov \omega} &  &  & \\
& & \ddots & & \\
& & & M_{\ov \omega}
\end{bmatrix}
\begin{bmatrix} \mathbf 0_2 & \mathbf I_2 &  & & \\
&  \mathbf 0_2 & \mathbf I_2 &  & \\
& & \ddots & \ddots & \\
& & & \mathbf 0_2
\end{bmatrix}
\qquad\text{where }
\mathbf 0_2 = \begin{bmatrix} 0 & 0 \\
0 &0 \end{bmatrix}.
\]
Hence, $T$ is nilpotent, $T^{k-1}$ has all zero entries except $2\times 2$ upper right block $M_{\ov \omega^{k-1}}$. Thus, $T^{k-1}$ has rank $2$ and $T^k=0$. This shows the claim. Moreover, the fact that $D_1$ and $D_2$ commute implies that for any $m\in \Z$,
\begin{equation}\label{jon5}
||J^{-m} (D_2)^m||=||(D_1)^{-m}||=1.
\end{equation}
By Lemma \ref{eq}, expansive matrices $J$ and $D_2$ are equivalent.

Define the matrix $B$ such that $S^{-1}BS$ is a block diagonal matrix where each Jordan block $J$ of $S^{-1}AS$ is replaced by a matrix $D_2$. 
This procedure is done for complex eigenvalues $\lambda$ with corresponding Jordan blocks of the form \eqref{jon}. If $\lambda \in \R$ is a real negative eigenvalue, then we replace $k \times k$ Jordan block $J$ by $-J$. Finally, if $\lambda$ is a positive eigenvalue, then we do nothing to $J$. 

By the construction we have defined a block diagonal matrix $S^{-1}BS$ for which two complex Jordan blocks of size $k$ for a conjugate pair $\lambda$ and $\ov \lambda$ of complex eigenvalues of $A$ correspond to two Jordan blocks of size $k$ for the eigenvalue $|\lambda|$ of $B$. In the case of a negative eigenvalue $\lambda$ of $A$, a Jordan block of size $k$ corresponds to a Jordan block of $B$ of the same size, but for positive eigenvalue $|\lambda|$. This shows that (ii)-(iv) hold.

Finally, to prove (i) observe that \eqref{jon5} implies that
for any $m\in \Z$,
\[
||S^{-1}A^{-m}B^m S||=||(S^{-1}A S)^{-m} (S^{-1} B S)^m||=1
\]
Hence, 
\[
\sup_{m\in \Z} ||A^{-m}B^m|| \le ||S|| ||S^{-1}||<\infty.
\]
By Lemma \ref{eq}, $A$ and $B$ are equivalent. 
\end{proof}

We are now ready to give the proof of Theorem \ref{thm:Construct-A_t}.

\begin{proof}[Proof of Theorem \ref{thm:Construct-A_t}]
Let $A$ be an expansive matrix. Let $B$ be the equivalent matrix with all positive eigenvalues which is guaranteed by Lemma \ref{l6.7}. By \cite[Theorem 1]{10.2307/2036109} there exists a real matrix $X$ such that $B=\exp(X)$. Since all eigenvalues of $B$ are $>1$, all eigenvalues of $X$ are positive. For $t\in \R$ define one-parameter group of dilations $A_t=\exp(P\ln t)$, $t>0$, where $P=\frac{1}{c} X$ and $c=\tr(X)$. By Lemma \ref{cf0} dilations $A_t$, $t>1$, are all equivalent with $B=A_{c}$, which in turn is equivalent with $A$.

Finally, the uniqueness of $P$ follows from Lemma \ref{cf2}. Indeed, suppose that $\{A'_t = \exp(P' \ln t)\}_{t>0} $ is another one-parameter group of dilations with a generator $P'$ satisfying (i) and such that $B$ is equivalent to $A'_t$ for some $t>1$. By Lemma \ref{cf0}, $B$ is equivalent to $A'_t$ for all $t>1$. Choose $t_0>1$ such that
\[
|\det(A)|=\det(B)= \det (A'_{t_0}) = \exp(\ln t_0 \tr(P'))= t_0.
\]
Since $A_{t_0}$ and $A'_{t_0}$ are equivalent, have all positive eigenvalues, and $\det (A_{t_0})=\det (A'_{t_0}) $, by Lemma \ref{cf2} we have $A_{t_0}=A'_{t_0}$. Likewise, by Lemma \ref{cf0}, $A_t$ and $A'_t$ are equivalent, have all positive eigenvalues, and $\det (A_{t})=\det (A'_{t})$ for all $t>1$. Thus, $A_t=A'_t$ for all $t>0$, which shows the uniqueness.
\end{proof}

As an immediate corollary of Theorem \ref{thm:Construct-A_t} we have the following result. A similar result to Corollary \ref{cor:Construct-A_t} was observed by Cheshmavar and F\"uhr in \cite[Remark 7.11]{chesh-fuhr}.
\begin{cor}\label{cor:Construct-A_t}
Let $\rho_1$ and $\rho_2$ be the quasi-norms associated to dilations $A_1$ and $A_2$, respectively. Let $P_1$ and $P_2$ be the generators of one-parameter groups of dilations as in Theorem \ref{thm:Construct-A_t} corresponding to $A_1$ and $A_2$, respectively. Then, $\rho_1$ and $\rho_2$ are equivalent if and only if $P_1=P_2$.
\end{cor}

We can now state the main result of our paper.
\begin{thrm}\label{thm:Hp-Hp} Let $A$ be an anisotropic dilation. Then, there exists an expansive continuous group $\{ A_t \}_{t > 0}$ such that its generator has all positive eigenvalues and discrete and continuous anisotropic Hardy spaces coincide $H_A^p = H_{\{ A_t \}}^p$. That is, for $f \in \mathcal{S}'$, we have
	\[ \| f \|_{H_A^p} \simeq \| f \|_{H_{\{ A_t \}}^p}. \] 
\end{thrm}
With this theorem, we are able to associate an anisotropic Hardy space $H_A^p$ with the parabolic PDE \eqref{eq:CT}, in the sense that the fundamental solution to \eqref{eq:CT} can be used as a kernel in the radial maximal characterization of $H_{\{ A_t \}}^p$.  In light of Theorem \ref{eqq} it is tempting to conclude that Theorem \ref{thm:Construct-A_t}  already accomplishes this. However, Theorem \ref{thm:Construct-A_t} implies only that for \textit{each} $t > 1$,  $H_A^p$ and $H_{A_t}^p$, both as anisotropic Hardy spaces (with respect to dilations $A$ and $A_t$), are equivalent. What we need is slightly stronger: the anisotropic Hardy space $H_A^p$, defined via discrete maximal functions, is the same as the parabolic Hardy space $H_{\{ A_t \}}^p$, defined via a continuous maximal function as shown below. 

\begin{lem}\label{dc}
 Let $\{ A_t \}_{t > 0}$ be an expansive continuous group and let $A=A_{t_0}$, where $t_0>1$. Then for any
$f \in \mathcal{S}'$, we have
\[ \| f \|_{H_{A}^p} \simeq \| f \|_{H_{\{ A_t \}}^p}. 
\] 
\end{lem}

\begin{proof}
Without loss of generality, by rescaling we can assume that generator $P$ of $\{A_t\}_{t>0}$ satisfies $\tr(P)=1$.
We start with the inclusion $H_A^p \supseteq H_{\{ A_t \}}^p$. Let $f \in H_{\{ A_t \}}^p$, that is, $\tilde{M}_{\varphi} f \in L^p$. To show that $f \in H_A^p$, by Theorem \ref{sf} it suffices to establish the pointwise inequality
	\begin{align}\label{est:rad-nont}
	 {M}_{\varphi}^0 f(x) \leq \tilde{M}_{\varphi} f(x).
	\end{align}
Note that for simplicity, we majorize the (discrete) radial maximal operator with a non-tangential (continuous) maximal operator, which follows immediately from the definitions. Indeed, we convert the continuous dilation (on the right) to the discrete dilation (on the left) by setting $t = (t_0)^k$, where $k \in \Z$. Since $b=|\det A|= t_0$ we have
	\[ b^{-k} \varphi(A^{-k} y) = \tilde{\varphi}_{(t_0)^k} (y). \]
Taking the supremum over $k\in \Z$ on the left-hand side and over $t > 0$ and $\tilde \rho(x - y) < t$, yields \eqref{est:rad-nont}.

For the reverse inclusion, let  $N \in \N$ be large enough so that if $f\in H^p_A$, then $M^0_{\mathcal{F}} f \in L^p$, where
\[
\mathcal{F} = \mathcal{F}_N = \{ \varphi \in \mathcal{S} : \| \varphi \|_{\alpha,N}:= \sup_{x\in \R^n} |\partial^\alpha \varphi(x) |(1+|x|)^N \leq 1, \  |\alpha| \leq N \}.
\]
Consider a continuous variant of radial and maximal function given by
\[
\tilde{M}^0_{\mathcal{F}} f(x) =\sup_{\varphi \in \mathcal F} \sup_{t>0} |f * \tilde \varphi_t(x)|. 
\] 
In light of Theorem \ref{sf}, it suffices to show that for any $\varphi \in \mathcal{F}$, there exist positive constants $c_1$ and $c_2$, independent of $f$, such that
	\begin{align}\label{est:Mcont-Maniso}
	\tilde{M}_{\varphi} f(x) \leq c_1 \tilde{M}^0_{\mathcal{F}} f(x) \leq c_2 M^0_{\mathcal{F}} (x).
	\end{align}

To establish the first inequality in \eqref{est:Mcont-Maniso} note that for any  $x\in \R^n$ we have
\[
\begin{aligned}
\tilde{M}_{\varphi} f(x) &= \sup_{t>0} \sup_{\tilde \rho(x-y)<t}  |f * \tilde \varphi_t(y)| 
= \sup\{ |f * \tilde \varphi_t(x+A_t z) |: t>0, \ \tilde \rho(z)<1 \}
\\
&= \sup\{ |f * \tilde \phi_t(x) |: t>0, \ \phi(\cdot)=\varphi (\cdot + z) \text{ for some }\tilde \rho(z)<1 \}.
\end{aligned}
\]
The semi-norms of $\phi$ can be crudely estimated as
\[
||\phi||_{\alpha,N} = \sup_{x\in \R^n} |\partial^\alpha \varphi(x) |(1+|x-z|)^N
\le 2^N(1+|z|)^N \sup_{x\in \R^n} |\partial^\alpha \varphi(x) |(1+|x|)^N \le 2^N(1+|z|)^N.
\]
Taking supremum over $z\in \R^n$ such that $\tilde \rho(z)<1$ shows that $\phi \in c_1 \mathcal F$ for some constant $c_1$, which yields the first inequality.

To establish the second inequality in \eqref{est:Mcont-Maniso}, we first show that it holds if $t \in [1, t_0]$ and then we extend it to all possible values of $t>0$.  Fix $\varphi \in \mathcal F$. For $x \in \R^n$ and $t \in [1, t_0]$, we write
	\[ |(f \ast \tilde{\varphi}_t)(x)| = \bigg| \int_{\R^n} f(x - z) t^{-1} \varphi(A_t^{-1} z) dz \bigg| 
	= |f \ast \psi(x)|, \]
	where $\psi(z) = t^{-1} \varphi(A_t^{-1}z)$. Observe that by chain rule, the partial derivatives $\p^{\alpha} \psi$, are controlled by $\p^{\beta} \varphi$, where $|\beta|\le |\alpha|$ as well as  norms of matrices $A_t^{-1}$, where $t \in [1, t_0]$. Indeed, by the chain rule, see \cite[Lemma 5.5]{pva}, there exists a constant $C>0$ such that 
	\[ 
	\begin{aligned}
	||\psi||_{\alpha,N} &\le C t^{-1} ||(A_t)^{-1}||^{|\alpha|} \sup_{|\beta| \le |\alpha|} \sup_{x\in \R^n} |\p^{\alpha}\phi ((A_t)^{-1}x)|(1+|x|)^N
	\\
&= C t^{-1} ||(A_t)^{-1}||^{|\alpha|} \sup_{|\beta| \le |\alpha|} \sup_{x\in \R^n} |\p^{\alpha}\phi (x)|(1+|A_t x|)^N \le  C t^{-1} ||(A_t)^{-1}||^{|\alpha|} ||A_t||^{N}.
\end{aligned}
	\]
Since $t\in [1,t_0]$ we deduce that $\psi \in c\mathcal{F}$ with $c$ depending on $t_0$.

Now let $t > 0$ be arbitrary. Let $k \in \Z$ be such that $t \in [(t_0)^k, (t_0)^{k + 1}]$. If we define $\tilde{t} = t/(t_0)^k$, then $\tilde{t} \in [1, t_0]$ and we have the identities $A_{t} = A_{\tilde{t} (t_0)^k} = A_{\tilde{t}} A_{(t_0)^k}$. Then,
	\begin{align*}
	f \ast \tilde{\varphi}_t (x) &= \int_{\R^n} f(x - z) t^{-1} \varphi((A_t)^{-1} z) dz = \frac{1}{\tilde{t}} \int_{\R^n} f(x - z) (t_0)^{-k} \varphi((A_{\tilde{t} (t_0)^k})^{-1} z) dz \\
	&= \frac{1}{\tilde{t}} \int_{\R^n} f(x - z) {\psi}_{k} (z) dz = \frac{1}{\tilde{t}} (f \ast \psi_{k})(x),
	\end{align*}
where $\psi(z) = \varphi((A_{\tilde{t}})^{-1} z)\in c\mathcal{F}$ and $\psi_k(z) = b^{-k} \psi(A^{-k}z)$. Therefore, taking the supremum of the continuous dilation over $t > 0$, with respect to a test function $\varphi \in \mathcal{F}$, is equivalent to taking the supremum of the discrete dilation over $k \in \Z$, with respect to another test function $\psi \in c\mathcal{F}$. Hence, we obtain the second inequality of \eqref{est:Mcont-Maniso}. This completes our proof.
\end{proof}

\begin{proof}[Proof of Theorem \ref{thm:Hp-Hp}]
Let $A$ be an expansive dilation. By Theorem \ref{thm:Construct-A_t} we can find an associated continuous expansive group $\{A_t\}_{t>0}$ such that its generator has all positive eigenvalues and $A$ is equivalent with dilation $A_{t_0}$ for some/all $t_0>1$. By Theorem \ref{eqq}, discrete anisotropic Hardy spaces $H^p_A$ and $H^p_{A_{t_0}}$ coincide with equivalent quasi-norms. Hence, we obtain the required conclusion by Lemma \ref{dc}.
\end{proof}

\section{Equivalence of dilations up to linear transformations}\label{S4}
In this section, we provide the classification of anisotropic Hardy spaces by correcting and expanding the results shown  in \cite{MR1982689}. 
The following result was shown by the first author \cite[Theorem 10.3]{MR1982689}.

\begin{thrm}\label{t10.3}
Let $\rho_1$ and $\rho_2$ be the quasi-norms associated to dilations $A_1$ and $A_2$, respectively. If $\rho_1$ and $\rho_2$ are equivalent, then for all $r>1$ and all $m=1,2,\ldots$
\begin{equation}\label{sp}
\spa \bigcup_{|\la|=r^\epsilon} \ker(A_1 - \la \mathbf I)^m =
\spa \bigcup_{|\la|=r} \ker(A_2 - \la \mathbf I)^m,
\end{equation}
where
\begin{equation}\label{ve}
\ve=\ve(A_1,A_2)=\ln |\det A_1|/\ln |\det A_2|.
\end{equation}
In \eqref{sp}, we treat $A_1$ and $A_2$ as linear maps on $\C^n$ and $\lambda$ varies over their complex eigenvalues.
\end{thrm}

In \cite{MR1982689} it was incorrectly claimed the converse to Theorem \ref{t10.3} also holds. Cheshmavar and F\"uhr \cite[Remark 7.4]{chesh-fuhr} have given an example showing that the converse is actually false. To illustrate this, we can use Corollary \ref{cor:Construct-A_t}: two dilations are equivalent exactly when their corresponding generators (of trace 1) are exactly the same. Consider the following example.

\begin{example}\label{2c02} For any $c\in \R$, consider $2\times 2$ dilation
\[
A_c = \begin{bmatrix} 2 & c \\
0 & 2 \end{bmatrix}.
\]
One can easily compute that
\[
A_c = \exp\bigg( \begin{bmatrix} \ln 2 & c/2 \\
0 & \ln 2 \end{bmatrix} \bigg).
\]
Hence, the generator $P_c$ of a one-parameter group of dilations from Theorem \ref{thm:Construct-A_t} corresponding to $A_c$ is given by
\[
P_c = \begin{bmatrix} 1/2 & c/(2\ln 2) \\
0 & 1/2 \end{bmatrix}.
\]
By Corollary \ref{cor:Construct-A_t}, dilations $A_c$ are not equivalent to each other for different choices of $c\in \R$. Obviously, the choice of $c=0$ corresponds to the classical isotropic setting. In general, by Lemma \ref{cf2}, matrices of the form
\begin{equation}\label{2star}
\begin{bmatrix}
2 & * & * & * \\
  & 2 & * & * \\
  &   &  \ddots  & * \\
  &  & & 2 \\
  \end{bmatrix}
\end{equation}
are equivalent if and only if all entries above the diagonal are identical.
\end{example}

Example \ref{2c02} suggests that it is rare when two dilations are equivalent. The situation changes drastically when we identify dilations up to a similarity. In this scenario dilations $A$ and $S^{-1}AS$, where $S\in GL_n(\R)$ are not distinguished. Hence, we are interested in equivalence of quasi-norms up to a linear transformation, see \cite[Definition 10.9]{MR1982689}. As a consequence of Corollary \ref{pn} we will see that the number of non-equivalent quasi-norms corresponding to $n\times n$ matrices of the form \eqref{2star} is actually finite and equal to the partition function $p(n)$.

This is a consequence of the following result, see \cite[Theorem 10.10]{MR1982689}. Since the original proof of this fact relied on incorrect formulation of \cite[Theorem 10.5(ii)]{MR1982689} due to the above mentioned problem with the converse of Theorem \ref{t10.3}, we need to give a corrected proof.

\begin{thrm}\label{t10.10}
Suppose we have two dilations $A_1$ and $A_2$ on $\R^n$. The following are equivalent:
\begin{enumerate}
\item
the quasi-norms $\rho_1$ and $\rho_2$ associated to $A_1$ and $A_2$, respectively, are equivalent up to a linear transformation, i.e., there is a constant $c>0$ and an invertible $n\times n$ matrix $S$ such that
\[
1/c \rho_1(x) \le \rho_2(Sx) \le c \rho_1(x) \qquad\text{for all }x\in\R^n.
\]

\item
for all $r>1$ and $m=1,2,\ldots$ we have
\begin{equation}\label{cdim}
\sum_{|\la|=r^\ve} \dim \ker(A_1 - \la \mathbf I)^m =
\sum_{|\la|=r} \dim \ker(A_2 - \la \mathbf I)^m,
\qquad\text{where } \ve = \frac{\ln |\det A_1|}{\ln |\det A_2|}.
\end{equation}
\end{enumerate}
\end{thrm}

\begin{proof}[Proof of Theorem \ref{t10.10}]
Suppose that two quasi-norms $\rho_1$ and $\rho_2$ are equivalent up to a linear transformation. Note that $\rho_2(S\cdot)$ is a quasi-norm associated with the dilation $S^{-1}A_2S$ since
\[
\rho_2(S(S^{-1}A_2Sx))=|\det A_2| \rho_2(Sx)= |\det(S^{-1}A_2S)| \rho_2(x).
\]
Hence, the quasi-norms  $\rho_1$ and $\rho_2$ are equivalent up to a linear transformation if and only if $A_1$ is equivalent to $S^{-1}A_2S$ for some $S\in GL_n(\R)$.

Since the quasi-norms $\rho_1$ and $\rho_2(S\cdot)$ are equivalent, Theorem \ref{t10.3} implies that for any  $r>1$, $m=1,2,\ldots$
\begin{multline*}
\spa \bigcup_{|\la|=r^\epsilon} \ker(A_1 - \la \mathbf I)^m
=
\spa \bigcup_{|\la|=r} \ker(S^{-1}A_2S - \la \mathbf I)^m
\\
= \spa \bigcup_{|\la|=r}  \ker(S^{-1}(A_2-\la \mathbf I )^mS)
= S^{-1}\bigg(\spa \bigcup_{|\la|=r}  (\ker(A_2-rId)^m)\bigg).
\end{multline*}
Hence, \eqref{cdim} holds.

It remains to show the converse implication (ii) $\implies$ (i). By Lemma \ref{l6.7} there exist expansive dilations $B_1$ and $B_2$ with positive eigenvalues which are equivalent to $A_1$ and $A_2$, respectively. Since original dilations $A_1$ and $A_2$ satisfy \eqref{cdim}, by Lemma \ref{l6.7}(iv)  their positive eigenvalue counterparts $B_1$ and $B_2$ satisfy
\[
\dim \ker(B_1 - r^\ve \mathbf I)^m =
 \dim \ker(B_2 - r \mathbf I)^m,\qquad\text{for all }r>1, m=1,2,\ldots
\]
Let $\rho_1'$ and $\rho_2'$ be quasi-norms associated to $B_1$ and $B_2$, respectively. Since $\rho_i$ and $\rho'_i$ are equivalent for $i=1,2$, it suffices to show that $\rho'_1$ and $\rho'_2$ are equivalent up to a linear transformation.

Using Theorem \ref{thm:Construct-A_t} we can rescale one of dilations, say $B_1$, to an equivalent dilation $B_1'$ so that  $\det B'_1 = \det B_2$, without affecting the conclusion (i). More precisely, we consider the unique one-parameter group of dilations $\{\exp(P\ln t)\}_{t>0}$  such that $B_1=\exp(P\ln t_1)$ for $t_1=\det B_1>1$, where $P$ is the generator as in Theorem \ref{thm:Construct-A_t}. Define $B_1'= \exp(P\ln t_2)$, where $t_2=\det B_2$. Then, $B_1'$ is equivalent to $B_1$ and $\det B'_1 = \det B_2$.   

Moreover, we claim that
\begin{equation}\label{jc}
\dim \ker(B'_1 - r \mathbf I)^m =
 \dim \ker(B_2 - r \mathbf I)^m
 \qquad\text{for all }r>1, m=1,2,\ldots
\end{equation}
Indeed, if the normal Jordan form of $P$ has a Jordan block of size $k$ corresponding to an eigenvalue $\lambda>0$, then $\exp(tP)$ has in its normal Jordan form a block of the same size corresponding to an eigenvalue $e^{\lambda t}$ for any $t\in \R$. In other words, 
\[
\dim \ker(P - \lambda \mathbf I)^m =
 \dim \ker(\exp(P \ln t)  - t^\lambda \mathbf I)^m
 \qquad\text{for all }t,\lambda>0, m=1,2,\ldots
\]
Take any $r>1$ and write it as $r=(t_2)^\lambda$ for some $\lambda>0$. Since $\ve=\ln t_1/\ln t_2$ we have $r^\ve = (t_1)^\lambda$. Hence,
\[
\dim \ker (\exp(P\ln t_2)- r \mathbf I)^m =
\dim \ker(P - \lambda \mathbf I)^m
=\dim \ker(\exp(P\ln t_1) - r^\ve \mathbf I)^m.
\]
This shows \eqref{jc}.

Finally, observe that for any matrix $B$, the number of Jordan blocks of size $\ge m$ corresponding to an eigenvalue $r$ is equal to
\[
\dim \ker(B - r \mathbf I)^m - \dim \ker(B - r \mathbf I)^{m-1}.
\]
Hence, \eqref{jc} implies that the number of Jordan blocks of size $m$ corresponding to an eigenvalue $r$ is the same for both $B_1'$ and $B_2$. Therefore, the matrices $B'_1$ and $B_2$ have the same Jordan normal form. In other words, there is an invertible $n\times n$ matrix $S$ such that $B'_1=S^{-1}B_2S$. This implies that the quasi-norm associated to $B_1'$, which is equivalent to $\rho_1'$, is $\rho'_2(S\cdot)$. This proves that quasi-norms $\rho_1'$ and $\rho_2'$ are equivalent up to linear transformations and so are $\rho_1$ and $\rho_2$. 
\end{proof}

Theorem \ref{t10.10} implies the following classification of expansive dilations according to their Jordan normal form.

\begin{cor}\label{pn} For any $k\le n$, take any sequences
\begin{align}\label{s1}
n_1, \ldots, n_k \in \N & \qquad n_1+ \ldots + n_k=n,
\\
\label{s2}
1<\la_1< \ldots <\lambda_k<2 & \qquad \lambda_1^{n_1} \cdots \lambda_k^{n_k}=2.
\end{align}
and partitions $\boldsymbol \pi^i$, $i=1,\ldots,k$ of $n_i$, i.e.,
\begin{equation}\label{s3}
\boldsymbol \pi^i = ( \pi^i_1 \ge  \ldots \ge \pi^i_{m_i} ) \in \N^{m_i}
\qquad \pi^i_1 +  \ldots + \pi^i_{m_i}=n_i.
\end{equation}
For specified parameters \eqref{s1}, \eqref{s2}, and \eqref{s3}, define the corresponding block diagonal matrix
\begin{equation}\label{s4}
A(\lambda_1,\ldots,\lambda_k; \boldsymbol \pi^1 , \ldots, \boldsymbol \pi^k )
\end{equation}
consisting of Jordan blocks for eigenvalues $\la_i$, $i=1,\ldots,k$, and sizes $\pi^i_j$, $j=1,\ldots, m_i$.  Then, any expansive dilation $A$ in $\R^n$ is equivalent up to a linear transformation to some dilation $A(\lambda_1,\ldots,\lambda_k; \boldsymbol \pi^1 , \ldots, \boldsymbol \pi^k )$. Moreover, this correspondence is $1$-to-$1$. That is, dilations of the form \eqref{s4} for distinct choices of eigenvalues $\lambda_1,\ldots,\lambda_k$ and partitions $\boldsymbol \pi^1 , \ldots, \boldsymbol \pi^k$ are not equivalent up to linear transformations.
\end{cor}

\begin{proof} By Theorem \ref{thm:Construct-A_t}, we can assume replace $A$ by an equivalent dilation $B$ with all positive eigenvalues. In addition, by rescaling we can assume that $\det B=2$. Then, $B$ has a Jordan normal form \eqref{s4} for some eigenvalues \eqref{s2} and the corresponding Jordan blocks of sizes given by partitions \eqref{s3}. Now, if we modify any of the partitions $\boldsymbol \pi^i$, $i=1,\ldots,k$, or any of the eigenvalues $\lambda_i$, then we necessarily change the value of 
\[
\dim \ker(A(\lambda_1,\ldots,\lambda_k; \boldsymbol \pi^1 , \ldots, \boldsymbol \pi^k ) - \la \mathbf I)^m
\]
for some $\lambda>1$ and $m=1,2,\ldots$  Theorem \ref{t10.10} guarantees that two different dilations of the form $A(\lambda_1,\ldots,\lambda_k; \boldsymbol \pi^1 , \ldots, \boldsymbol \pi^k )$ are not mutually equivalent up to linear transformation.
\end{proof}

Recall that two anisotropic Hardy spaces $H^p_{A_1}(\R^n)$ and $H^p_{A_2}(\R^n)$ are equivalent up to linear transformations if there exists an invertible $n\times n$ matrix $S$ such that the dilation operator $f \mapsto f(S^{-1}\cdot)$ defines an isomorphism between $H^p_{A_1}(\R^n)$ and $H^p_{A_2}(\R^n)$, see \cite[Definition 10.9]{MR1982689}. This happens precisely if $A_1$ and $S^{-1}A_2S$ are equivalent dilations, see the proof of \cite[Theorem 10.10]{MR1982689}. As a consequence of Corollary \ref{pn} dilations of the form \eqref{s4} classify anisotropic Hardy spaces $H^p_A(\R^n)$ up to linear transformations. 

\begin{thrm}
Suppose we have two dilations $A_1$ and $A_2$ on $\R^n$. The following are equivalent:
\begin{enumerate}
\item
$A_1$ and $S^{-1}A_2S$ are equivalent for some $n\times n$ invertible matrix $S$, 
\item
\eqref{cdim} holds for all $r>1$ and $m=1,2,\ldots$,
\item 
$H^p_{A_1}(\R^n)$ and $H^p_{A_2}(\R^n)$ are equivalent up to linear transformations for all $0<p\le 1$,
\item
$H^p_{A_1}(\R^n)$ and $H^p_{A_2}(\R^n)$ are equivalent up to linear transformations for some $0<p\le 1$.
\end{enumerate}
\end{thrm}

\begin{remark}
A similar classification result holds for homogeneous Besov spaces which were originally introduced in \cite{besov}. According to \cite[Corollary 6.5]{chesh-fuhr} two anisotropic Besov spaces associated to dilations $A_1$ and $A_2$ are the same if and only if quasi-norms corresponding to transposes $(A_1)^T$ and $(A_2)^T$ are equivalent. Consequently, one can show that (unweighted) anisotropic Besov spaces are classified up to linear transformations by dilations of the form \eqref{s4}.
\end{remark}

\begin{remark}
Cheshmavar and F\"uhr have introduced the concept of coarse equivalence of dilations in terms of their quasi-norms \cite{chesh-fuhr}. In light of Theorem \ref{t10.10}, it might be tempting to introduce a coarse equivalence of quasi-norms up to linear transformations. However, it is not difficult to show that this notion coincides with the above concept of equivalence up to linear transformation. Indeed, as a corollary of \cite[Theorem 7.9(b)]{chesh-fuhr}, any two coarsely equivalent dilations are equivalent up to linear transformation.
\end{remark} 

\begin{remark}
As another consequence of Corollary \ref{pn} we can deduce that for any choice of
\[
1<\lambda_1 < \ldots< \lambda_k<\infty, \ k\le n,
\]
there exist only finitely many equivalence classes (up to linear transformations) of $n\times n$ dilation matrices $A$ with above magnitudes of eigenvalues. Indeed, there exist only a finite choice of multiplicities \eqref{s1} with corresponding finite number of choices of partitions \eqref{s3} that produce the required representatives
\[
A(\lambda_1/c,\ldots,\lambda_k/c; \boldsymbol \pi^1 , \ldots, \boldsymbol \pi^k ),
\qquad\text{where }
c=(\lambda_1^{n_1} \cdots \lambda_k^{n_k})^{1/n}/2^{1/n}.
\]
\end{remark}

As a consequence, we obtain the following corollary.

\begin{cor}
The number of equivalence classes of anisotropic Hardy spaces $H^p_A(\R^n)$, $0<p\le 1$, up to linear transformations for dilations $A$, which have only one eigenvalue, equals the partition function $p(n)$.
\end{cor}

\section{Further Connections between Hardy Spaces and PDEs} \label{S5}


In this section, we discuss open questions on the Hardy spaces associated with differential operators. The relationship between the anisotropic Hardy space and many variants of the Hardy spaces associated to operators \cite{MR3108869}, \cite{MR2163867}, \cite{MR2868142}, \cite{MR2481054} is not clear, given the time-component in the associated differential equation \eqref{eq:CT}. Related to this issue, we can formulate a new Hardy space adapted to the quasi-norm of a continuous group, which does have a natural formulation related to a pseudo-differential operator $L$. Denoting this second Hardy space by $H_L^p$, we ask if such a space is well-defined with respect to the norm used, and whether the semigroup $\{e^{-tL}\}_{t>0}$ satisfies Davies-Gaffney estimates. 

\subsection{Parabolic Setting} To set the context for these questions, we fix an expansive continuous group $\{ A_t \}$, which defines the associated parabolic differential equation \eqref{eq:CT}: 
	\begin{align*} \frac{\p u}{\p t} = \frac{1}{t} \cdot (D_t^{-1} \Delta D_t) u = L_t u.
	\end{align*}
In the frequency domain, the fundamental solution $\Phi$ is given in a simple form. 
\begin{prop}[{\cite[Section 1.3]{MR0417687}}] 
\label{fund}
Let ${\Phi} \in \mathcal{S}$ be defined by
		\begin{equation}
		\label{Phi} \hat{{\Phi}} (\xi) = \exp[-4\pi^2 \langle B \xi, \xi \rangle], 
		\end{equation}
		where $B$ is a norm-inducing matrix $B$ for $\{A_t^*\}$.
	Then $\tilde \Phi_t(x)=t^{-1} \Phi(A_t^{-1}x)$ satisfies the differential equation \eqref{eq:CT}. Moreover, if $f \in \mathcal{S}'$, then $u(x, t) = f \ast \tilde{\Phi}_t (x)$ also satisfies the same equation.
	\end{prop}
Proposition \ref{fund} is an elementary result, but we include the proof for completeness. Its proof does not require the assumption \eqref{CT} made in \cite{MR0417687}. However, it does require the following property of a norm-inducing matrix $B$ for $\{A_t^*\}$, which we state again: 
	\begin{align}\label{eq:B-property}
	\frac{d}{dt} \langle B A_t^* x, A_t^* x \rangle = \frac{1}{t} \langle (BP^* + PB) A_t^* x, A_t^* x \rangle =  \frac{1}{t} \langle  A_t^* x, A_t^* x \rangle. 
	\end{align}

	\begin{proof} If $\tilde \Phi_t$ is a solution for the PDE \eqref{eq:CT}, then by taking the Fourier transform we obtain
	\begin{align}\label{eq:PDE2}
\frac{\partial}{\partial t}[\hat{\Phi}(A_t^* \xi )] = \frac{-4\pi ^2 \langle A_t^* \xi , A_t^* \xi  \rangle}{t} \hat{\Phi}(A_t^* \xi).
	\end{align}
If $\Phi$ is given by \eqref{Phi}, then a simple calculation using \eqref{eq:B-property} shows that it satisfies \eqref{eq:PDE2}. Conversely, observe that for fixed $\xi \in \R^n$, \eqref{eq:PDE2} can be seen as an ordinary differential equation of the form
\[
\frac{d}{dt} h(t) =  \frac{-4\pi ^2 \langle A_t^* \xi , A_t^* \xi  \rangle}{t} h(t).
\]
Hence, its solution must be of the form $h(t) = c_0 e^{s(t)}$, where $s'(t) = -4\pi^2 \langle A_t^* \xi , A_t^* \xi \rangle/t$. By the property \eqref{eq:B-property}, we have $s(t) = (-4\pi^2) \langle B A_t^* \xi, A_t^* \xi \rangle$. Since $h(t)=\hat \Phi(A_t^*\xi)$, the fundamental solution $\Phi$ satisfies $\hat{\Phi}(\xi) = \exp[-4\pi^2 \langle B \xi, \xi \rangle]$. Consequently, if $f \in \mathcal{S}'$, then $F : \R^n \times (0, \infty)\rightarrow \C$, defined by $F(x, t) = f \ast \tilde{\Phi}_t(x)$, is also a solution to \eqref{eq:CT}.
	\end{proof}

Now observe that the parabolic PDE, associated with the operator $L_t = (D_t^{-1} \Delta D_t)/t$, depends on $t$, for which there is no viable semigroup theory. Indeed, if we naively set $T_t f = \tilde{\Phi}_{t} * f$, a computation with the Fourier transform gives
    \begin{align*}
	(T_t T_r f)^{\wedge}(\xi) = \exp(-4\pi^2 \langle \underbrace{A_{\frac{t}{t + r}} BA_t^* + A_{\frac{r}{t + r}} BA_r^*}_J )\xi, A_{t + r}^* \xi \rangle ).
	\end{align*}
We will have semigroup exactly if the $J$-term satisfies the identity 
$A_{\frac{t}{t + r}} BA_t^* + A_{\frac{r}{t + r}} BA_r^* = B A_{t + r}^*$, which is not likely. So we cannot make sense of the operation $e^{-t L_t} f$, nor use $H_{L_t}^p$ to denote $H_{\{ A_t \}}^p$. This leads to the following problem. 
Given a dilation matrix $A$ and its associated expansive group $\{ A_t \}$, and having established the equivalence between the anisotropic and parabolic Hardy spaces $H_A^p \simeq H_{\{ A_t \}}^p$, investigate the connection with the theory of Hardy associated with operators.

\subsection{Alternative Parabolic Approach}\label{sect:alternative}
By seeking a semigroup structure associated with the expansive group $\{ A_t \}$, we formulate an alternative definition of Hardy space. To do this, we fix a continuous group $\{ A_t \}_{t>0}$. Let $\tilde{\rho}_*$ be the quasi-norm associated with $\{A_t^*\}$, and define the alternative fundamental solution 
\begin{equation}\label{Phi2}
\hat{\Phi}(\xi) = \exp(-4\pi^2 \rho_*^2 (\xi)).
\end{equation}
Then we have a semigroup $T_t f = \tilde{\Phi}_{\sqrt{t}} \ast f(x)$, as made apparent by the Fourier transform and the homogeneity property $\tilde{\rho}_* (A_t^* \xi) = t \tilde{\rho}_* (\xi)$, 
        \begin{align*}
            \widehat{T_t T_s f} (\xi) 
                = \exp (-4\pi^2 \tilde{\rho}_*^2 (A_{\sqrt{t}}^* \xi) ) \exp (-4\pi^2 \tilde{\rho}_*^2 (A_{\sqrt{s}}^* \xi)) \hat{f}(\xi) &= \exp(-4\pi^2 \rho_*^2 (\xi \sqrt{s + t} )) 
                \\
                &= \widehat{T_{s + t} f} (\xi). 
        \end{align*}
Furthermore, $\Phi$ given by \eqref{Phi2} is the fundamental solution of $\p_t u = Lu$, defined in frequency by
	\begin{align}\label{eq:pde-alt}
	 \p_t \hat{u}(\xi) = \widehat{Lu}(\xi) = -4\pi^2 \tilde{\rho}_* (\xi)^2 \hat{u}(\xi).
	 \end{align}
The operator $L$ is the infinitesimal generator of the semigroup $\{T_t\}$, which is defined formally for $f$ in the domain of $L$, by $Lf = \lim_{t \rightarrow 0^+} \frac{T_t f - f}{t}$. 
Then we are now in a position to define the space $H_L^p$ as all tempered distributions $f \in \mathcal{S}'$ such that 
	\[ \mathcal{M}_{\Phi}^0 f = \sup_{t > 0} |e^{-tL} f| = \sup_{t > 0} |\tilde{\Phi}_{\sqrt{t}} \ast f| \in L^p. \] 
We can now attempt to place the pseudo-differential operator $L$ from \eqref{eq:pde-alt} among existing literature. We start with a basic question concerning the nature of these $H_L^p$ spaces. Given $\{ A_t \}$ and its dual $\{ A_t^* \}$, we can have more than one homogeneous norm $\rho^*$ to use in defining the fundamental solution $\hat{\Phi}$. For example, even in the diagonal case, fix $\delta > 0$, and define
	\[ A_t = 
			\begin{pmatrix}
			t^{\delta}		&	 	0 \\
			0									&		t^{\delta} 		
			\end{pmatrix}
	\text{ with }  
		 P = 
		\begin{pmatrix}
		\delta 	&		0 \\
		0					& \delta
		\end{pmatrix}.
	\]
Since $A_t = A_t^*$, we do not need to make a distinction between the homogeneous norms $\rho$ or $\rho^*$. 
Associated with this group are two natural choices of quasi-norms.
The first is the canonical quasi-norm, from solving the identity $|A_t^{-1} x| = 1$, which gives the norm $\tilde{\rho}_1 (x) = |x|^{1/\delta}$. The second is given by $\tilde{\rho}_2(x) = \sqrt{ |x_1|^{2/\delta} + |x_2|^{2/\delta}}$, so that it satisfies the homogeneity $\tilde{\rho}(A_t x) = t\tilde{\rho}(x)$. 
Their geometries do differ: $\tilde{\rho}_1 (x) = 1$ exactly when $x$ is on the boundary of the Euclidean unit ball, while $\tilde{\rho}_2 (x) = 1$ for $x$ on the boundary of $\ell^{2/\delta}$ unit ball given by $\| x \|_{\ell^{2/\delta}}^{2/\delta} = \sum_{j = 1}^2 |x_j|^{2/\delta} = 1$. 

A fundamental question is to consider whether the resulting Hardy spaces, from the two norms, agree. Consider the example when $\delta = 2$, with the homogeneous norms $\tilde{\rho}_1 (\xi) = |\xi|^{1/2}$ and $\tilde{\rho}_2 (\xi) = \sqrt{|\xi_1| + |\xi_2|}$. Their respective fundamental solutions are $ \widehat{\Phi^{(1)}} (\xi) = \exp(-4\pi^2 |\xi|)$ and $\widehat{\Phi^{(2)}} (\xi) = \exp(-4\pi^2 |\xi_1|) \exp(-4\pi^2 |\xi_2|)$, which under the inverse Fourier transform, are given by the Poisson kernel  $\Phi^{(1)} (x) = P_2(x)$ and the product of two Poisson kernels in $\R$, which we denote by $Q(x) = \Phi^{(2)} (x) = P_1(x_1) P_1(x_2)$, where $P_n(x)=c_n(1+|x|^2)^{-(n+1)/2}$. We proceed along the classical arguments, see \cite[Chapter III.1]{stein93} or \cite[Chapter 2.1]{MR2463316}. 

Observe that the first norm leads to the classical Hardy space. Denote the Hardy space associated with the kernel $Q$ to be $H_Q^p$, with the norm $\| f \|_{H_Q^p} = \| \mathcal{M}_Q^0 f \|_{p}$ for $f \in \mathcal{S}'$, defined formally for $f \in \mathcal{S}'(\R^2)$ for which $f \ast Q_s (x)$ is defined for $s > 0$.
We can readily establish the inclusion $H^p \subseteq H_Q^p$ by decomposing $Q$ along each variable, and obtain
	\begin{align*}
	Q(x) = Q_1 (x_1) Q_2 (x_2)
		&= \sum_{j, k = 0}^{\infty} 2^{-(k +j)} 2^{-k} \varphi^{(j, k)} \bigg(\frac{x_1}{2^j}, \frac{x_2}{2^k} \bigg),  
	\end{align*}
where $\varphi^{(j, k)}(x_1, x_2) = \Phi^{(k)}(x_1) \Phi^{(j)}(x_2)$, and $\Phi^{(k)}$ are smooth cutoff functions in $\R$ and  bounded in $\mathcal{S}(\R^2)$. Then we can majorize the radial maximal function of $Q$ with respect to a grand maximal function, and obtain the inclusion that $H^p \subseteq H_Q^p$. 

However, the reverse inclusion is unknown. Classically, one defines a test function by
$\Psi(x) = \int_1^{\infty} \eta(s) P_s (x) dx$, 
where $\eta$ is smooth on $[1, \infty)$, and satisfies
	\[ \int_1^{\infty} \eta(s) ds = 1, \qquad \int_1^{\infty} s^k \eta(s) ds = 0 \text{ for } k \in \N. 
\]
The $\Psi$ is shown to be in $\mathcal{S}$, and majorizes the maximal function associated with the Poisson kernel. However, when we use $Q$ in this construction, $\Psi$ cannot be shown to be smooth, given the lack of differentiability of $\hat{Q}$ along the $\xi_1$ and $\xi_2$ axes. This leaves open the question of whether we do have $H_Q^p \subseteq H^p$, or, more generally, if two homogeneous norms to the same continuous group leads to the same Hardy space. It is worth adding that if $p>1$, then the Hardy space $H^p$ and $H^p_Q$ actually coincide with $L^p$ by \cite[Chapter II.4]{stein93}. In general, the fact the Hardy spaces $H^p_L$ coincide with $L^p$ spaces for $p>1$ is related to the boundedness of maximal functions along curves due to Stein and Wainger \cite{stein1978}. However, the following problem remains open.

\begin{question}
 Given an essentially continuous expansive group $\{ A_t^* \}$, with two homogeneous quasi-norms $\tilde{\rho}_1^*$ and $\tilde{\rho}_2^*$, consider the resulting fundamental solution and PDE, given by the differential operators $L_1$ and $L_2$, respectively. Do they result in the same Hardy space, that is, $H_{L_1}^p = H_{L_2}^p$, $0<p \le 1$?
\end{question}

 Lastly, given $H_L^p$, fixed by a specific choice of $\tilde{\rho}$ associated with $\{ A_t^* \}$, we naturally inquire where in the present literature this Hardy space is placed. Given the vast literature, we can narrow this question to verifying the Davies-Gaffney estimate as follows. Let $\{ A_t \}$ be an expansive continuous group, and consider $(\R^n, \tilde{\rho}, dx)$ as a space of homogeneous type. We seek to determine if the fundamental solution $\Phi$ of $L$ satisfies the off-diagonal Davies-Gaffney estimate \cite[Assumption 2.2]{MR2868142}.

\begin{question} For open subsets $U_1, U_2 \subset \R^n$ and $f_1, f_2$ supported on $U_1, U_2$, will we have, for all $t > 0$, 
	\begin{align*}
	 |\langle e^{-t L} f_1, f_2 \rangle| \leq C \exp \left( - \frac{\operatorname{dist} (U_1, U_2)^2}{ct} \right) \| f_1 \|_2 \| f_2 \|_2, 
	 \end{align*}
where $e^{-t L} f_1 = \tilde{\Phi}_{\sqrt{t}} \ast f$ and $\operatorname{dist}(U_1, U_2) = \inf \{ \tilde{\rho}(x - y) : x_1 \in U_1, x_2 \in U_2 \}$?
\end{question}

\bibliographystyle{plain} 		
	\bibliography{Cite}

\end{document}